\documentclass[11pt,a4paper]{article}
\usepackage{a4wide, amsmath, amssymb, amsthm, amscd}
\usepackage{graphicx,caption,color}
\usepackage[shortlabels]{enumitem}

\newcommand{\Nat}{\mathbb{N}}

\newcommand{\C}{\mathbb{C}}
\newcommand{\RS}{\widehat{\mathbb{C}}}
\newcommand{\tsim}[2][1.7]{
  \mathrel{\underset{#2}{\scalebox{#1}[1]{$\sim$}}}
}
\newcommand{\GOD}{\mathcal{D}} 
\newcommand{\brjuno}{\mathcal B}

\DeclareMathOperator{\Ar}{Area}    
\DeclareMathOperator{\supp}{supp}
\DeclareMathOperator{\inR}{\underline{r}}
\DeclareMathOperator{\outR}{\bar{r}}
\newtheorem{theorem}{Theorem}[section]
\newtheorem{lemma}[theorem]{Lemma}
\newtheorem{proposition}[theorem]{Proposition}
\newtheorem{corollary}[theorem]{Corollary}
\newtheorem*{theoremA}{Theorem A}
\newtheorem*{theoremB}{Theorem B}
\newtheorem*{theoremA'}{Theorem A'}
\newtheorem*{theoremB'}{Theorem B'}
\theoremstyle{definition}
\newtheorem{definition}[theorem]{Definition}
\newtheorem{remark}[theorem]{Remark}

\def\qedprop{\hfill{\hbox{%
  \hskip 1pt%
  \vrule width 7pt height 6pt depth 1.5pt%
  \hskip 1pt}} \vskip.3cm}

\author{
  \small N\'uria Fagella\footnote{Partially supported by spanish grants MTM2011-26995-C02-02 and MTM2014-52209-C2-2-P.}\\
  \small Dept.~de Mat.~Aplicada i An\`alisi \\
  \small Universitat de Barcelona \\
     \small Gran Via 585, 08007 Barcelona \\
  \small Catalunya, Spain \\
  \small \texttt{fagella@maia.ub.es}
  \and
  \small Christian Henriksen\footnote{Partially supported by the Danish Council for Independent Research grant DFF-4181-00502}\\
  \small Dept.~of Appl.~Math.~and Comp.~Science \\
  \small Technical University of Denmark \\
  \small Matematiktorvet \\
  \small Building 303 B, room 154 \\
  \small 2800 Kgs. Lyngby, Denmark\\
  \small \texttt{chrh@dtu.dk}
}

\title{The Fine Structure of Herman Rings}
\begin{document}
\maketitle

\abstract{We study the geometric structure of the boundary of Herman rings in a model family of Blaschke products of degree 3. Shishikura's quasiconformal surgery  relates the Herman ring to the Siegel disk of a quadratic polynomial. By studying the regularity properties of the maps involved, we can transfer McMullen's results on the fine local geometry of Siegel disks to the Herman ring setting. 
\renewcommand{\thefootnote}{}
\addtocounter{footnote}{-1}
}

\section{Introduction}

We consider the dynamical system induced by the iterates of a rational map $f: \RS \to \RS$ of degree $d\geq 2$, where $\RS$ denotes the Riemann sphere or compactified complex plane. We use the notation $f^n:=f\circ \overset{n)}{\cdots} \circ f $ to denote the $n^{th}$ iterate of $f$. Under this dynamics, the Riemann sphere splits into two completely invariant sets: the {\em Fatou set}, formed by those points for which the sequence $\{f^n\}$ is normal in some neighborhood; and its complement, the {\em Julia set}. By definition the Fatou set is open and therefore the Julia set is a compact set of the sphere. Connected components of the Fatou set, also known as {\em Fatou components}, map onto one another and are eventually periodic \cite{sul}.  The Julia set is the common boundary between the different Fatou components and, consequently, the dynamics on this set is chaotic. For background on the dynamics of rational maps we refer for example to \cite{carlesongamelin} and \cite{mil}.

An especially relevant particular case of rational maps are polynomials, which are exactly (up to M\"obius conjugation) those rational maps for which infinity is a fixed point and has no preimages other than itself. In particular this implies that infinity is a {\em superattracting fixed point}, and  the dynamics are locally conjugate to $z\mapsto z^d$ around this point for some $d\geq 2$, the degree of the polynomial; it also means that the basin of attraction of infinity, that is the set of points attracted to infinity under iteration, is connected and completely invariant. Therefore its boundary is compact in $\C$ and coincides with the Julia set of the polynomial. 

Periodic Fatou components of rational maps $f$ are completely classified \cite{fat}: a periodic component $U$ is either part of a  basin of attraction of an attracting or parabolic cycle, or a {\em rotation domain}, which means that  some iterate of $f|_U$ is conformally conjugate to a rigid rotation of irrational rotation number. Rotation domains may be simply connected and then they are called {\em Siegel disks}; or doubly connected,  in which case they are known as  {\em Herman rings}.  By definition, Herman rings separate the Julia set and have a disconnected boundary. 
 
The dynamics of a rational map is determined, to a large exent, by the orbit of  its {\em critical points}, i.e. the zeros of its derivative. Indeed, any basin of attraction must contain a critical point \cite{fat} and every boundary component of a rotation domain is accumulated by a {\em critical orbit} (i.e. the orbit of a critical point) \cite{fat,shi}. In some cases the relation is even stronger: any rotation domain with rotation number of bounded type, that is whose entries in its continued fraction are bounded (see Section \ref{setup}), have Jordan boundaries which actually {\em contain} a critical point  \cite{zhang}. In this work we will only consider rotation domains with this property.

Herman rings are undoubtably the least well known among all possible types of periodic Fatou components of rational maps, one reason being that they are not associated to any periodic point with a certain multiplier, as all other types are (basins of attraction or Siegel disks). Their closest relatives, Siegel disks, are much better understood, both in terms of conditions for their existence and in terms of the different properties that their boundaries possess. Relevant to our work will be, for instance,  the string of geometric results about the fine structure of Siegel disks, proven by McMullen in \cite{mcm}, like self-similarity of the Siegel disk around the critical point or measurable depth of the critical point in the filled Julia set (see Section \ref{prelimsSD}).   

But there is a procedure to relate Siegel disks and Herman rings,  known as {\em Shishikura's surgery} (see Section \ref{surg} and \cite{shi}). Roughly speaking, starting with a map that has a Herman ring $H$  (of a certain rotation number), this construction produces a map having a Siegel disk $S$ (of the same rotation number); at the same time it relates both functions via a quasiconformal map $\Phi$ which is a partial conjugacy between them. Intuitively, Shishikura's procedure erases the hole from the  Herman ring (and from all its preimages), substituting the dynamics here by a rigid rotation (see Figure \ref{figA}). The procedure is reversible, and it therefore ties certain problems about Herman rings (like existence for given rotation numbers) to the corresponding problem for Siegel disks. 

The quasiconformal map $\Phi$ mentioned above 
opens up a possibility to transfer geometric properties between Siegel disks and their corresponding Herman rings. Some of them are fairly obvious to transport: If the boundary of $S$ 
is a Jordan curve, so will be both boundaries of $H$; or they will all contain a critical point or none will. But other geometric properties are not necessarily preserved by general homeomorphisms or quasiconformal maps.

In this paper we study the extra regularity properties of the quasiconformal map $\Phi$ and use them to transfer some of McMullen's results about the fine geometry of Siegel disks to corresponding statements about Herman rings. These are the principal contents of Theorems A and B. In the latter, we additionally conclude that  the full boundary of the Herman ring is, surprisingly,  \emph{tightly similar}  to that of a Siegel disk, even though one of them is disconnected and the other is not. The concept of tight similarity, introduced here,  is stronger than regular similarity. In other words, zooming in around the critical point, the holes of the Herman ring tend to become invisible, until the Siegel disk and the Herman ring become indistinguishable form each other (see Figure \ref{figB}). 

Our study is done using a model family of rational maps of degree 3, which is one of the simplest that exhibits Herman rings of all rotation numbers. It also has the property that Shishikura's surgery relates it to the family of quadratic polynomials used in McMullen's results. However, in the same way that McMullen's properties also hold for quadratic-like mappings (maps which behave locally as a quadratic polynomial, see \cite{dhpolike}), our theorems also extend to appropriate rational-like maps (see Remark \ref{generalize}). 

\subsubsection*{Acknowledgements}
The authors are grateful to Carsten Lunde Petersen and Saeed Zakeri for their insightful comments. They also thank the Institut de Matemˆtica at Universitat de Barcelona and the {Department of Applied Mathematics and Computer Science of the Technical University of Denmark}  for their hospitality while this work was in progress.

%
%


\subsection{Setup and statement of results} \label{setup}

Arithmetics play an important role in dynamics as well. It is important to distinguish between three nested classes of irrational numbers.
For $a_1, a_2, \ldots \in \Nat$, we let
\[
  [a_1, a_2, \ldots] = \cfrac{1}{a_1 + \cfrac{1}{a_2 + \ddots}}
\]
denote the continued fractional expansion with $a_1, a_2, \ldots$ as coefficients, and denote the convergents by
\begin{equation} \label{convergents}
  \frac{p_n}{q_n} = [a_1, a_2, \ldots, a_n].
\end{equation}
See \cite{khi} for details. An irrational number $\theta$ is a \emph{quadratic irrational} if the sequence of coefficents $a_1, a_2, \ldots$ is eventually periodic.
The quadratic irrationals are exactly the irrational roots of quadratic equations with integer coefficients.

A more general set of irrational numbers are the irrationals of \emph{bounded type}.
They are numbers whose coefficients satisfy
$\sup a_n < \infty$.

An even more general class of irrationals is the class of  \emph{Brjuno numbers} which we denote by $\brjuno$. A Brjuno number is charaterized by the denominators of its convergents; a number is Brjuno if and only if $\sum \log q_{n+1} / q_n < \infty$.

These classes of irrational numbers are relevant to dynamics. If $f$ is a holomorphic map in a neighborhood of the origin, such that $f(0)=0$ and $f'(0)=  e^{2i\pi\alpha}$ with $\alpha \in \brjuno$, then 
there is a neighborhood of $0$ on which $f$ is conjugate to an irrational rotation of rotation $\alpha$ \cite{sie,bru,rus}. If the map is globally defined, this is part of a Siegel disk. Conversely, if a quadratic polynomial has an invariant Siegel disk, its center is a fixed point with multiplier $e^{2i\pi\alpha}$ with $\alpha\in\brjuno$ \cite{yoc}. 

Here and in the rest of the article, we fix an irrational number $\theta$ of bounded type. 
We let $\lambda = e^{2i\pi\theta}$, and fix the quadratic polynomial 
\begin{equation} \label{polyn}
P(z) = \lambda z + z^2.
\end{equation}
This polynomial has a unique critical point $\omega: = -\frac{\lambda}{2}$. The origin is a fixed point of multiplier $P'(0)= \lambda$.   We know that $P$ posseses a Siegel disk $S$ centered at $z=0$, because the numbers of bounded type form a subset of the Brjuno numbers.

On the rational end, we work with the simplest family that can exhibit Herman rings, namely
\[
  f_{a, b}(z) := b z^2 \frac{az+1}{z+a} ,
\]
for $a, b \in \C$.  Every $f_{a, b}$ has superattracting fixed points at the origin and at infinity.
Additionally, there are  two other criticial points which we denote by  $\omega_1$ and  $\omega_2$.

The family $f_{a, b}$ provides examples of Herman rings.
It is well known \cite{bru,shi,yoc}  that for any irrational $\alpha$, there exists $a, b$ such that $f_{a, b}$ has a Herman ring with rotation number $\alpha$, if and only if $\alpha$ is Brjuno.
In \cite{MR2190332}, Buff, Fagella, Geyer and Henriksen show that for a Brjuno number $\alpha$, there exists a pointed disk holomorphically embedded in the $a, b$ parameter space of $f_{a, b}$, such that every mapping in the disk possesses an invariant Herman ring with rotation number $\alpha$.  

From now on, we let $a, b$ be parameters chosen such that $f_{a, b}$ has an invariant Herman ring with rotation number $\theta$.
Since we shall not vary $a, b$, we drop the indexes and simply denote the rational map by $f$. Hence
\begin{equation} \label{rat}
 f(z) = b z^2 \frac{az+1}{z+a} ,
 \end{equation}
has a Herman ring of rotation number $\theta$, which we denote by $H$. 
The Herman ring has two boundary components $\partial^j H$, $j=1,2$, which are both quasicircles, each containing a critical point $\omega_j$ (this follows from results of Herman, Ghys, Douady, \'Swiatek and Shishikura, see e.g. \cite[Sections 7.2 and 7.3]{surbook}). 
We number the components such that $\partial^1 H$ is contained in the bounded component of the complement of $H$, and we number the critical points, such that $\omega_j \in \partial^j H$, $j=1,2$.

As mentioned, $f$ has a fixed critical point at infinty. Dynamically this means that infinity is a superattracting fixed point, the immediate attracting bassin of which we denote by $A_f(\infty)$.
The boundary of $A_f(\infty)$ is a proper subset of the Julia set $J(f)$.
This is in contrast to what happens for $P$, the quadratic polynomial in (\ref{polyn}), where the boundary of the basin of infinity, $A_P(\infty)$,  coincides with $J(P)$.

In \cite{MR854306}, Shishikura introduced a surgery that could turn a map with a cycle of Herman rings into a map with a cycle and Siegel disk and viceversa (see Section \ref{surg}). We will use a special case of his construction to show the first of our two main theorems.

\begin{theoremA}
 Let $\theta$ be of bounded type and let $P$ and $f$ be as in (\ref{polyn}) and (\ref{rat}) respectively. In the setup above,  there exists a $P$-invariant simply connected domain $D \Subset S$, and a quasiconformal mapping $\Phi : \RS \to \RS$, such that
  \begin{enumerate}[\rm (a)]
    \item $\Phi$ conjugates $P$ to $f$ on $\RS \setminus D$
    \item $\Phi$ maps $J(P)$ onto $\partial A_f(\infty)$, $\partial S$ onto $\partial^2 H$, and $\omega$ to $\omega_2$.
    \item $\bar{\partial} \Phi = 0$ a.e. outside $\GOD := \cup_{n=0}^\infty P^{-n}(D)$.
    \item $\Phi$ is $C^{1+\alpha}$-conformal at $\omega$ with $\Phi'(\omega)\neq 0$.
  \end{enumerate}
\end{theoremA}

The notion of $C^{1+\alpha}$-conformal is due to McMullen and is defined as follows.
\begin{definition}[$C^{1+\alpha}$ conformality]
  A mapping $\phi$ is $C^{1+\alpha}$-conformal at $z_0$ if there exists $\alpha, \delta, M > 0$  such that
  \[
    \phi(z) = \phi(z_0) + \phi'(z_0) (z - z_0) + R(z)
  \]
  where
  \[
    |R(z)| \leq M |z - z_0|^{1+\alpha}
  \]
  when $|z-z_0| < \delta$.

 Note that this is stronger that saying that $\phi$ is $\C-$differentiable at the point $z_0$. We say $\phi$ is $C^{1+\alpha}$-anticonformal at $z_0$, if $\bar{\phi}$ is $C^{1+\alpha}$-conformal at $z_0$.
\end{definition}

Theorem A is illustrated in Figure \ref{figA}.
\begin{figure}[hbt!]
\captionsetup{width=0.8\textwidth}
  \begin{center}
    \hfil \includegraphics[width=0.45\textwidth]{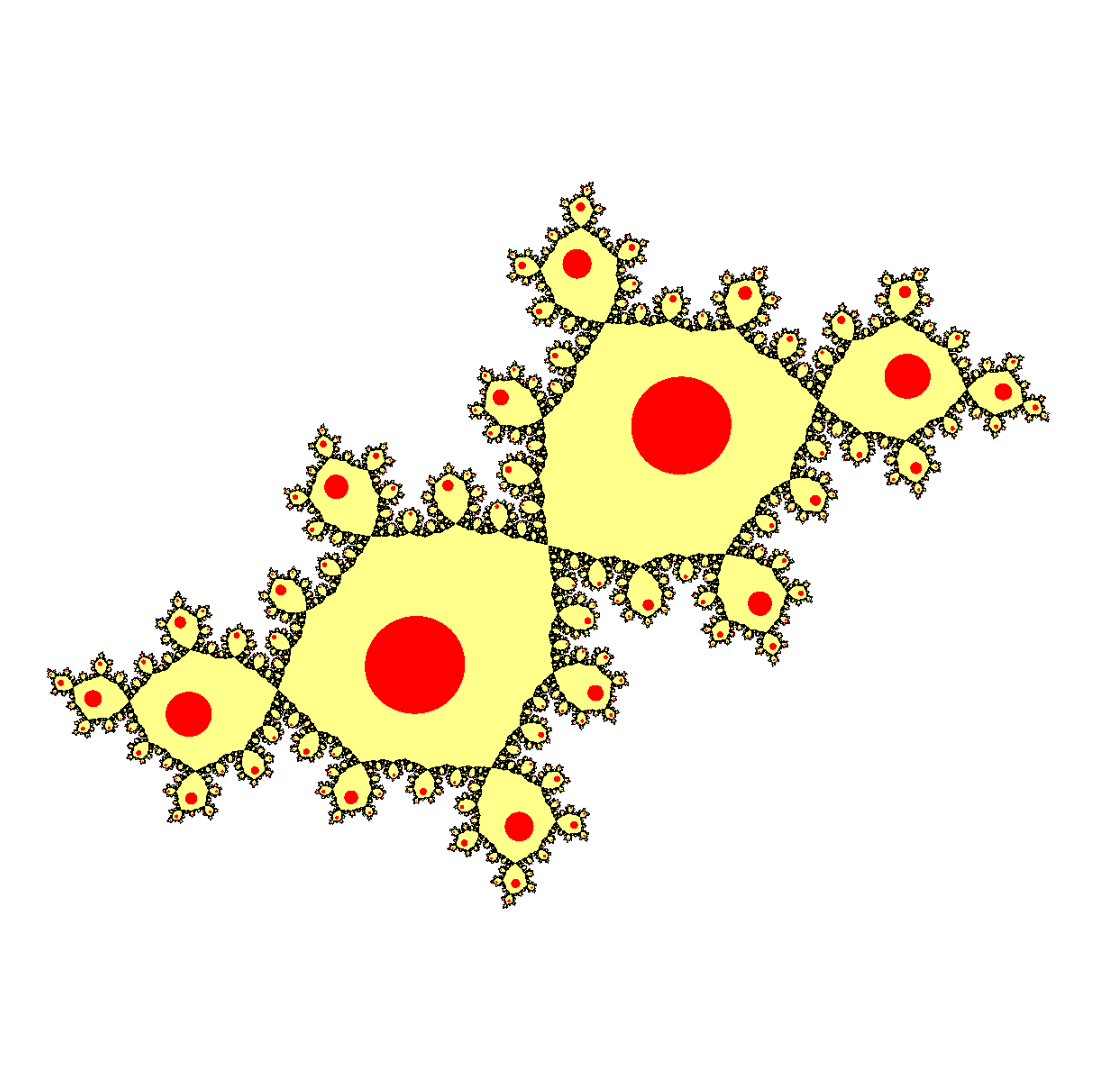} \hfil
    \includegraphics[width=0.45\textwidth]{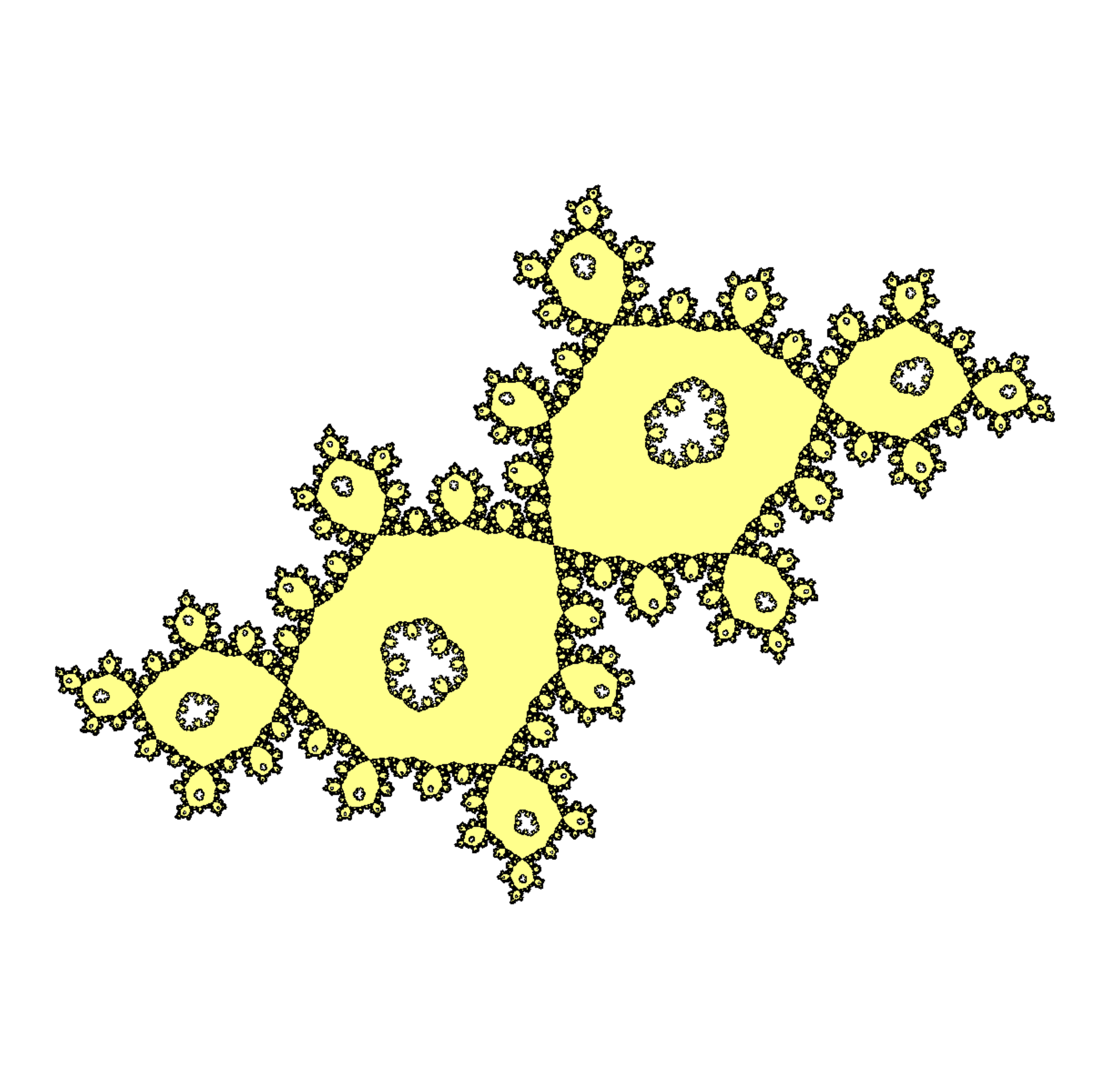} \hfil
  \setlength{\unitlength}{\textwidth}
  \put(-0.5,0.25){\vector(1,0){0.1}}
    \put(-0.46,0.26){$\Phi$}
    \put(-0.8,0.2){\small $S$}
    \put(-0.79,0.17){\scriptsize $D$}
    \put(-0.72,0.23){\scriptsize $\omega$}
    \put(-0.24,0.23){\scriptsize $\omega_2$}
    \put(-0.28,0.18){\small $H$}
    \put(-0.36,0.32){$A_f(\infty)$}
    \put(-0.84,0.32){$A_P(\infty)$}
    \put(-0.725,0.225){\circle*{0.01}}
    \put(-0.249,0.2225){\circle*{0.01}}
 \end{center}
\caption{\small We illustrate Theorem A for $\theta=\frac{\sqrt{5}-1}{2}$, the golden mean.
    The dynamics of $P$ is illustrated in the figure to the left, which is symmetric with respect to $\omega$.
    We have colored the completely invariant set $\GOD$ in red, and $J(P)$ in black.
    The dynamics of $f$ is illustrated in the figure to the right, which has been rotated, scaled and translated to illustrate the similarity with $P$.
    The Herman ring and its preimages are colored yellow and $J(f)$ black.
    By Theorem A, there exists a quasiconformal homeomorphism $\Phi$ which is conformal outside the red set. Outside $D$, $\Phi$ conjugates $P$ to $f$, and $\Phi$ is $C^{1+\alpha}$-conformal at $\omega$.} %
  \label{figA}
\end{figure}
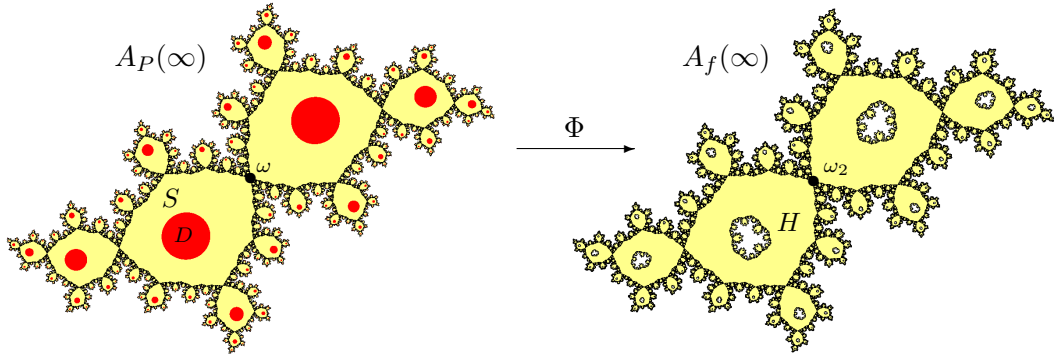
The fact that the quasiconformal conjugacy provided by Theorem A is $C^{1+\alpha}$ at $\omega$ allows us to see that the boundary components of $H$ are locally similar to the boundary of $S$.
To make a precise statement, we introduce a notion of similarity that is stronger than Tan Lei's notion of asymptotic similarity introduced in \cite{MR1086745}.  Let $B(c,r)$ denote the open ball of center $c\in\C$ and radius $r>0$.

\begin{definition}[Tight similarity] \label{tight}
  We say that two compact sets $A, B$ are \emph{tightly similar} at $z_0$ if there exists $\delta, \beta>0$ and $L>0$ such that
  \begin{enumerate}[\rm (1)]
    \item $a \in A \cap B(z_0, \delta) \Rightarrow d(a, B) \leq L |a-z_0|^{1+\beta}$
    \item $b \in B \cap B(z_0, \delta) \Rightarrow d(b, A) \leq L |b-z_0|^{1+\beta}$
  \end{enumerate}
  When $A$ and $B$ are tightly similar at $z_0$, we write
  \[ 
    A \tsim{z_0} B .
  \]
  A compact set $A$ is {\em tightly self-similar} around $z_0\in A$ if 
  \[
  A-z_0 \tsim{0} \kappa(A-z_0)
  \]
  for some $\kappa \in \C$, with $|\kappa| >1 $.
\end{definition}
 It is easy to check that the notion of tight similarity at $z_0$ is an equivalence relation on the compact subsets of $\C$. We are now ready to state the second main theorem.

\begin{theoremB}
 Let $\theta$ be of bounded type and let $P$ and $f$ be as in (\ref{polyn}) and (\ref{rat}) respectively. In the setup above, the following are satisfied.
  \begin{enumerate}[\rm (a)]
    \item There exists a scaling factor $L \in \C \setminus \{0\}$ such that
      \[
        L(J(P)-\omega) \tsim{0} \partial A_f(\infty) - \omega_2
        \tsim{0} J(f) - \omega_2.
      \]
    \item When $\theta$ is a quadratic irrational, $\partial^2 H$ is tightly self-similar around $\omega_2$.
    \item The Siegel disk $S$ admits an Euclidean triangle with vertex at $\omega$, if and only if $H$ admits Euclidean triangle with vertex at $\omega_2$.
  \end{enumerate}
\end{theoremB}

Theorem B is illustrated in Figure \ref{figB}. There we can see the simililary between $J(P), \partial A_f(\infty)$ and $J(f)$.
Even though $J(f)$ and $\partial A_f(\infty)$ are topologically very different, the components of $J(f) \setminus \partial A_f(\infty)$ have increasingly small diameters, and get increasingly close to $A_f(\infty)$   as we zoom in at $\omega_2$.

\begin{figure}[hbtp!]
\captionsetup{width=0.8\textwidth}
  \begin{center}
    \begin{tabular}{cc}
    \includegraphics[width=0.37\textwidth]{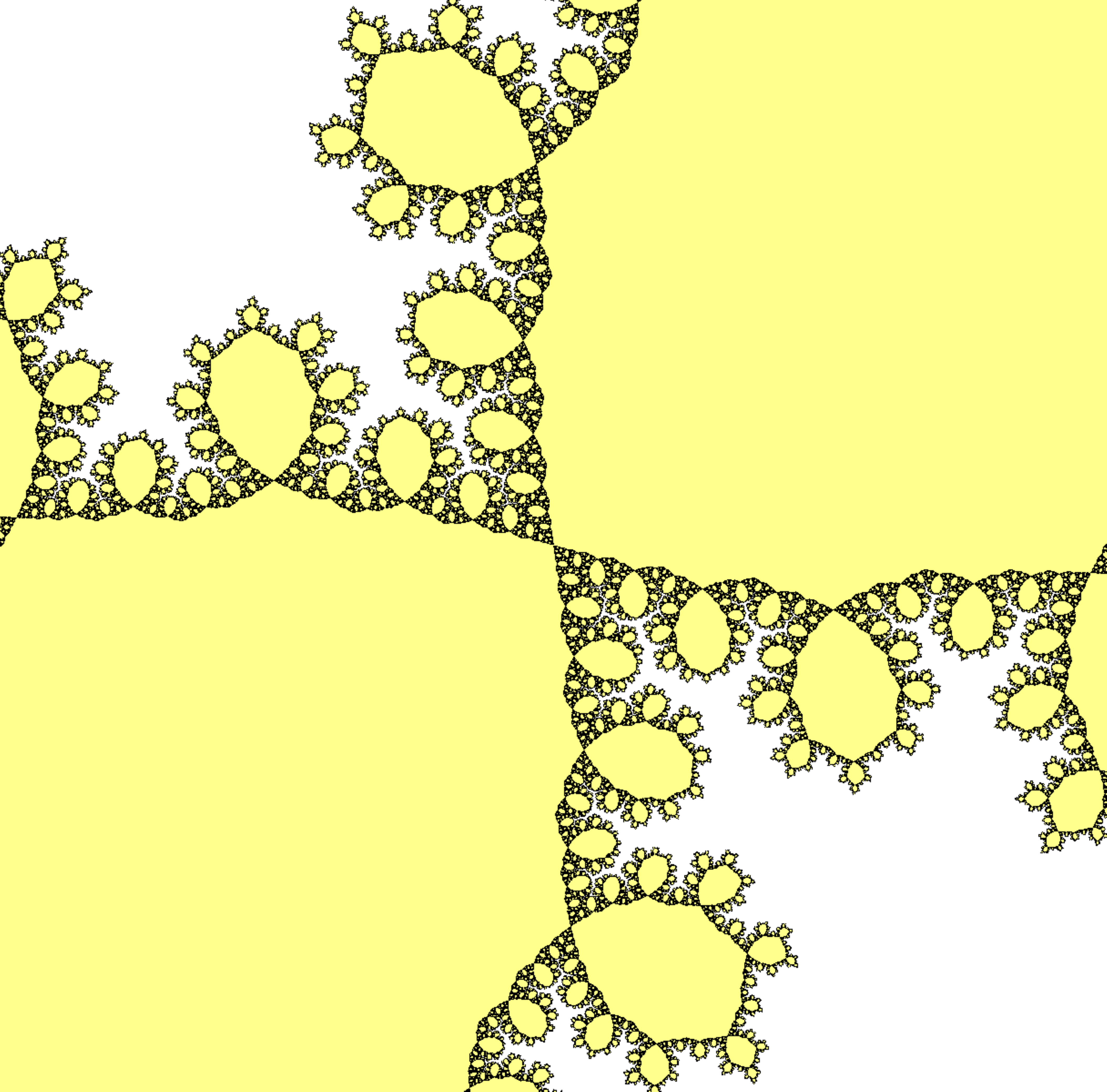} &\includegraphics[width=0.37\textwidth]{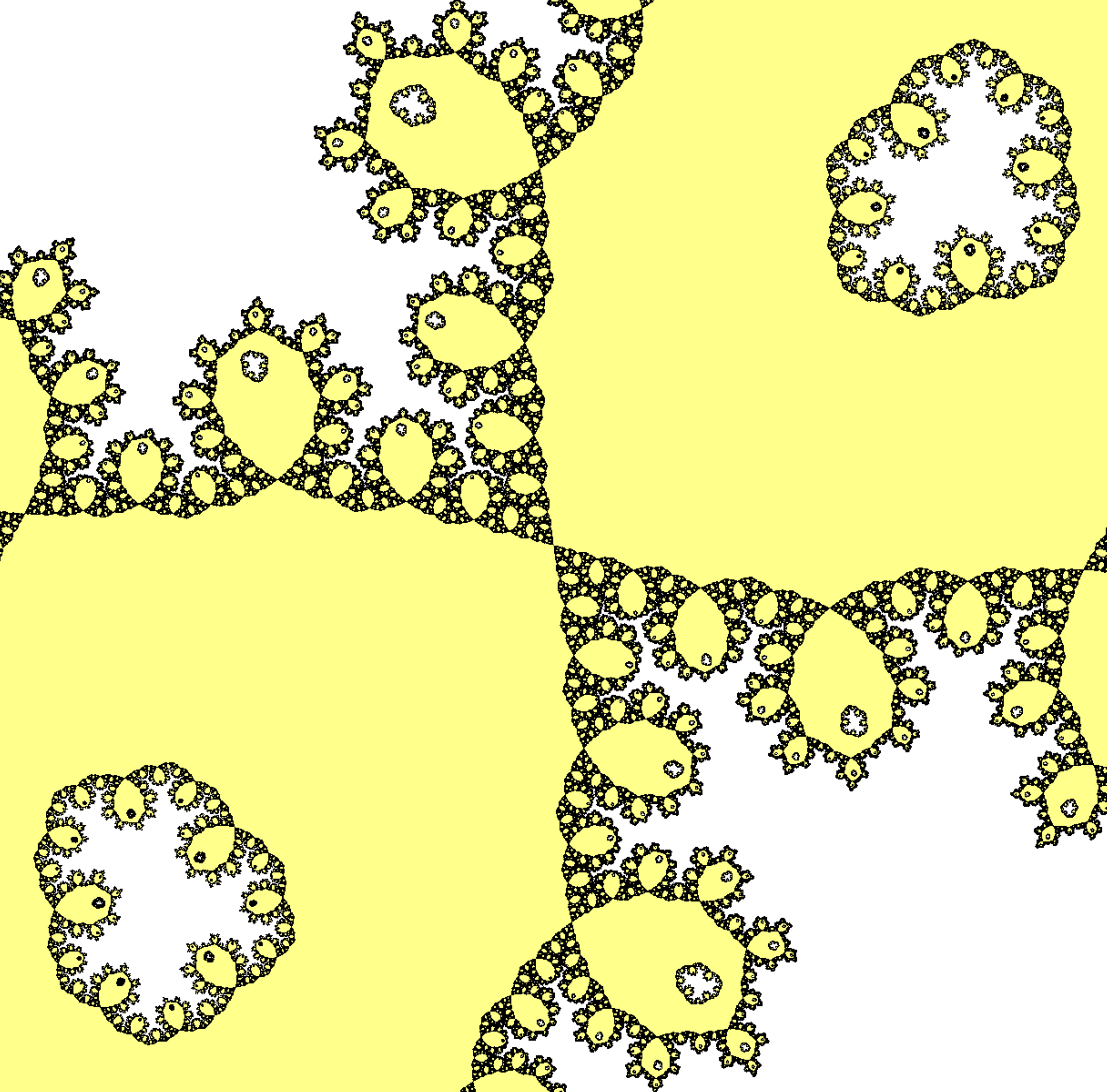}\\
    \includegraphics[width=0.37\textwidth]{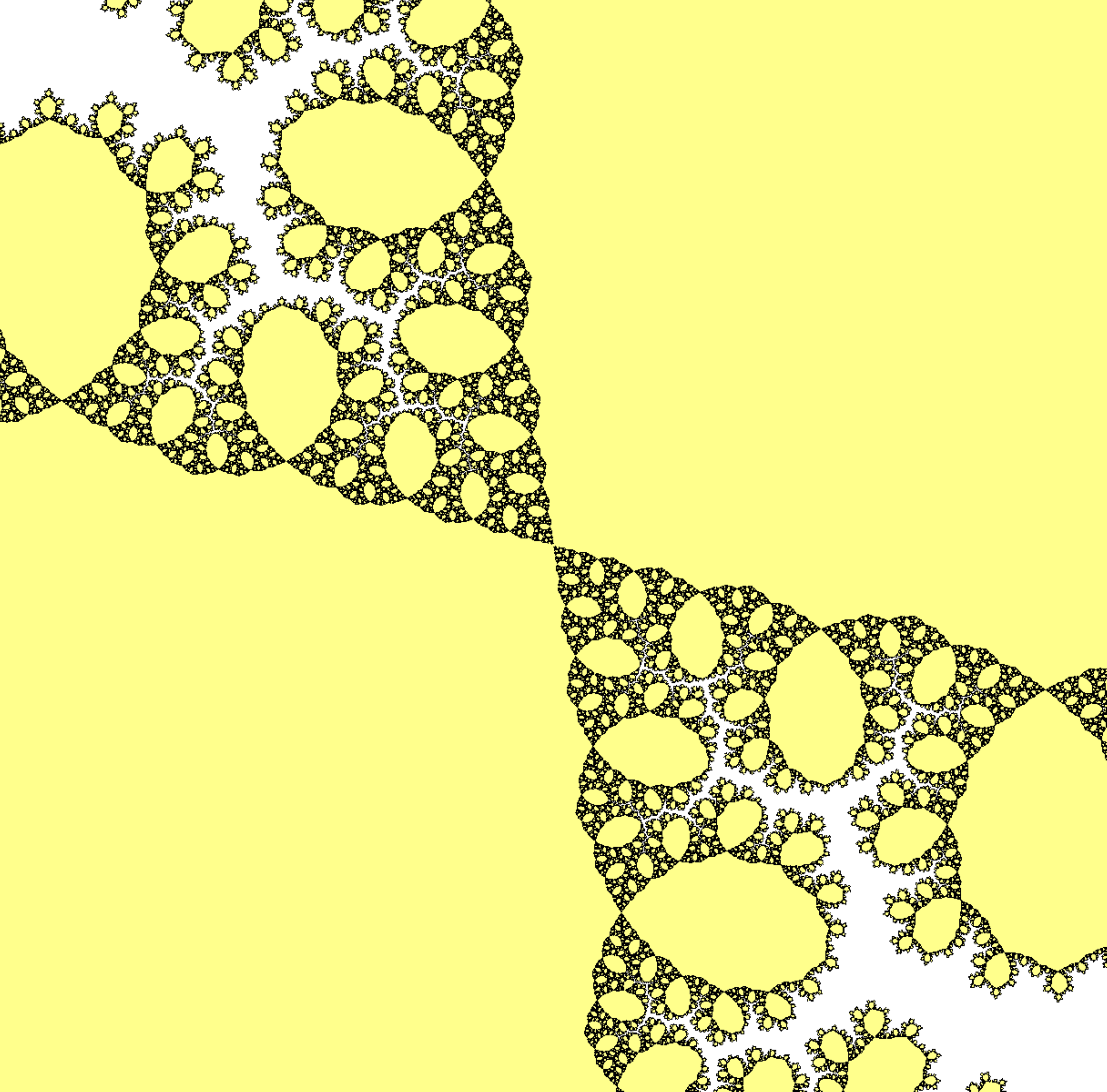} &\includegraphics[width=0.37\textwidth]{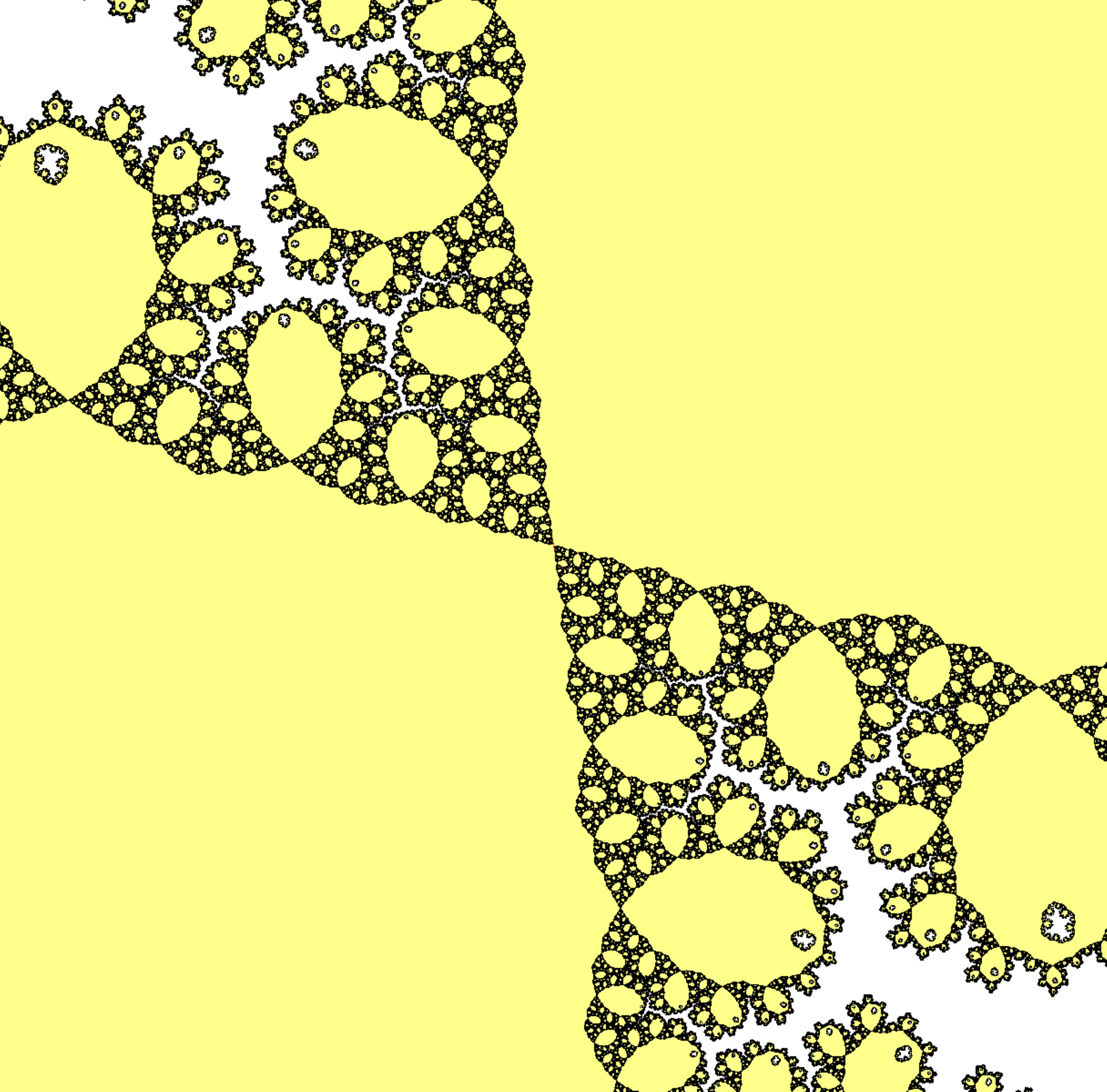}\\
    \includegraphics[width=0.37\textwidth]{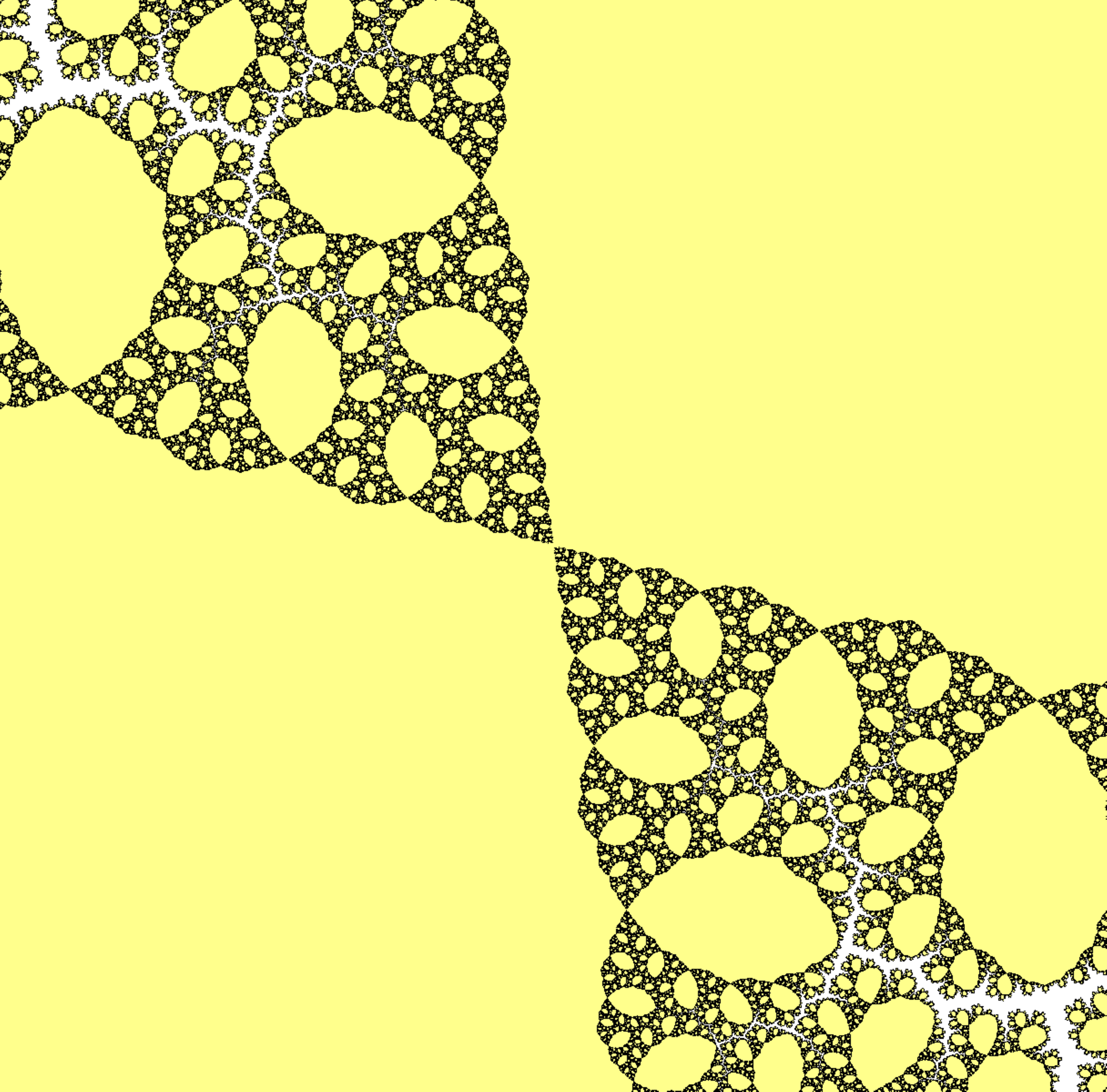} &\includegraphics[width=0.37\textwidth]{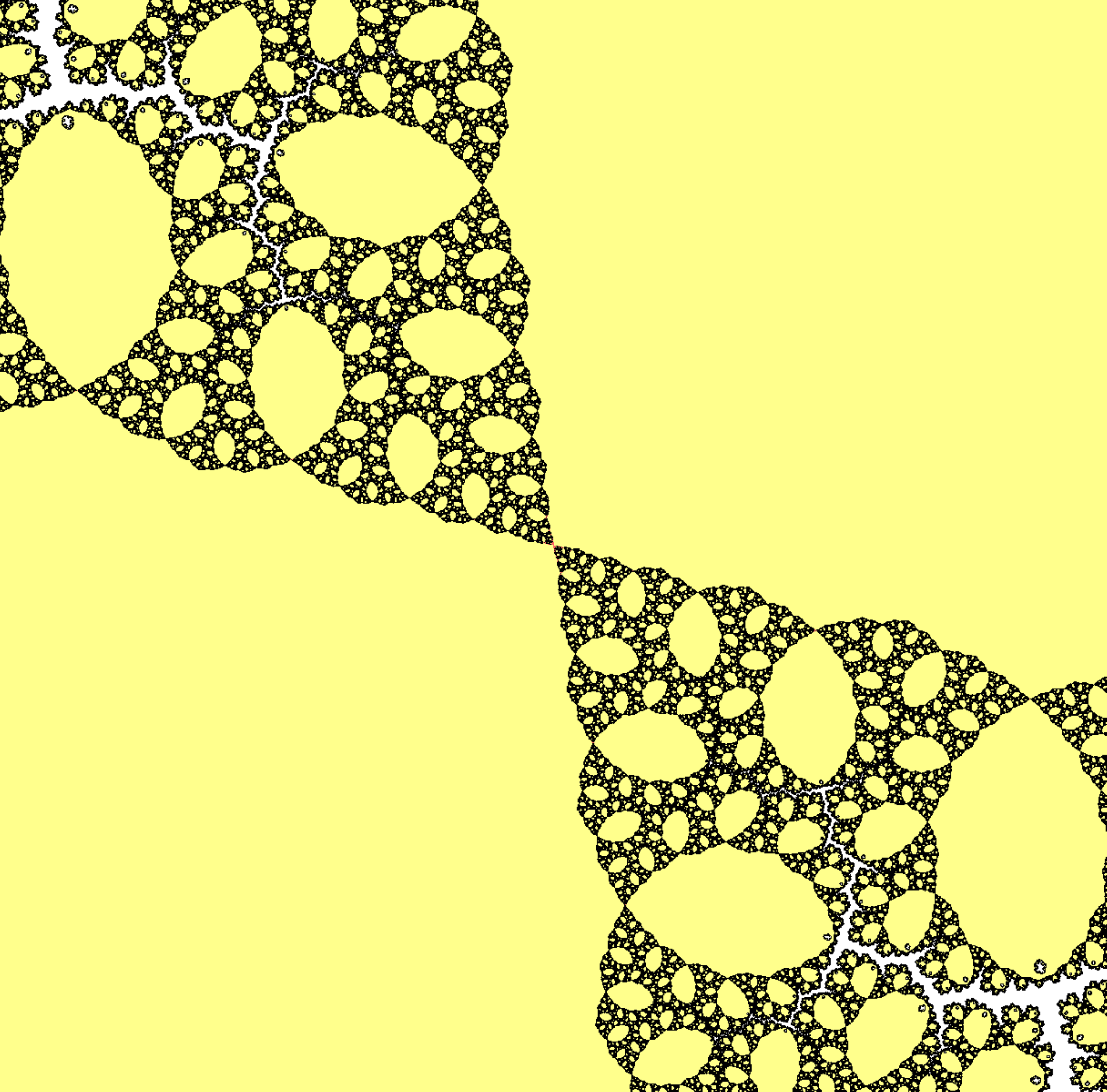}
    \end{tabular}
  \end{center}
  \caption{\small    \label{figB} On the left column  Julia sets of the quadratic polynomial $P$ as in (\ref{polyn}) with $\theta=\frac{\sqrt{5}-1}{2}$.
    To the right is illustrated the Julia set of $f$ as in (\ref{rat}), with $(a, b)$ chosen so that $\theta$ is as above.     Going down, we increase the zoom around $\omega$ and $\omega_2$ respectivley, and see that similarity becomes more and more pronounced.  } 
\end{figure}

Theorem B can be applied when searching for mappings in $f_{a, b}$ with Herman rings of bounded type rotation number.
Indeed, if $q_n$ denotes the denominator of the convergents to  $\theta$ defined in  (\ref{convergents}), from Theorem A we know

\begin{align*}
\frac{f^{q_{n+1}}(\omega_j)-\omega_j}{f^{q_n}(\omega_j)-\omega_j}
& =  \frac{\Phi\left(P^{q_{n+1}}(\omega)\right)-\Phi(\omega)}{\Phi\left(P^{q_n}(\omega)\right)-\Phi(\omega)}\\
& =   \frac{P^{q_{n+1}}(\omega)-\omega + \mathcal{O}\left(|P^{q_{n+1}}(\omega)-\omega|^{1+\alpha}\right)}
{P^{q_{n}}(\omega)-\omega + \mathcal{O}\left(|P^{q_{n}}(\omega)-\omega|^{1+\alpha}\right)}\\
& =   \frac{P^{q_{n+1}}(\omega)-\omega}{P^{q_{n}}(\omega)-\omega}
\left (1 + \mathcal{O}\left(|P^{q_{n}}(\omega)-\omega|^{1+\alpha}\right)\right)
\end{align*}
for $j = 1, 2$. 

Therefore the points $(a, b)$ where $f_{a, b}$ has a Herman ring with rotation number $\theta$ satisfy 
\[
  \frac{P^{q_{n+1}}(\omega)-\omega}{P^{q_n}(\omega)-\omega}
  \approx \frac{f_{a, b}^{q_{n+1}}(\omega_j)-\omega_j}
         {f_{a, b}^{q_n}(\omega_j)-\omega_j},
         ~~j=1,2,
\]
when $n$ is large, which narrows the search to a one-complex-dimensional set of parameters.

Let us finally note that we can use the dynamics to extend the results to any point $\omega'$ that under iteration goes to $\omega$, and any point $u$ that is eventually mapped to either $\omega_1$ or $\omega_2$.

\begin{corollary} \label{extended}
Suppose that $\omega'$ and $u$ satisfy $P^n(\omega')=\omega$ and $f^m(u) \in \left\{ \omega_1, \omega_2 \right\}$. Then,
\begin{enumerate}[\rm (a)]
    \item $\Phi$ is $C^{1+\alpha}$-conformal at $\omega'$.
    \item There exists a scaling factor $L' \in \C \setminus \{0\}$ such that
      \[
        L'(J(P)-\omega') \tsim{0} J(f) - u .
      \]
    \item If $\theta$ is a quadratic irrational, $H'$ is some connected component of $f^{-k}(H)$, $k\geq 0$, and $u \in \partial H'$, then $\partial H'$ is tightly self-similar around $u$.
    \item The preimage $P^{-n}(S)$ admits an Euclidean triangle with vertex at $\omega'$, if and only if $f^{-n}(H)$ admits Euclidean triangle with vertex at $u$.
\end{enumerate}
\end{corollary}

\begin{remark} \label{generalize}
McMullen's results extend to a much more general class than quadratic polynomials, namely to all quadratic-like maps with a fixed point of derivative $e^{2\pi i \theta}$, being $\theta$ and irrational of bounded type (see \cite[Theorem 5.1]{mcm}). Using this result, one can see  (not without some work) that Theorems A and B also extend to a more general setting than the model family considered above. More precisely, if $g$ is a rational like map that straightens to a member of the model family $f_{a,b}$ having a Herman ring of rotation number $\theta$ (an irrational of bounded type) then our results apply to the small Julia set of $g$. 
\end{remark}

\section{Preliminaries about $P$ and its Siegel disk} \label{prelimsSD}

Before proving our two main theorems, we review some facts about the Julia set $J(P)$ and the Siegel disk $S$.

Recall that $P(z) = e^{2i\pi\theta}z + z^2$, with $\theta$ of bounded type, and $\omega=-\frac{e^{2\pi i \theta}}{2}$, the critical point. In \cite{MR1440932}, Petersen showed the following theorem.

\begin{theorem} \label{Petersen}
The Julia set $J(P)$ is a locally connected set of zero Lebesgue measure. 
\end{theorem}

McMullen proved a string of geometric results in \cite{mcm}.
He showed that $J(P)$ has Hausdorff dimension strictly less than two.
He also showed that $\omega$ is a density point of the filled Julia set of $P$. In fact,
he showed an even stronger result, namely that $\omega$ is a measurable deep point in a subset of the {\em filled in Julia set} $K(P):=\C\setminus A_P(\infty)$.

\begin{definition}[(Measurable) deep point]
Let $z_0\in \C$ and $E$ be a Borel set. For $r>0$, let $s(r)$ be the largest radius so that $B(z,s(r)) \subset B(z_0,r)\setminus E$, for some $z\in \C$. We say that $z_0$ is a {\em deep point} in $E$ if there exists $\alpha>0$ such that 
\[
s(r)\leq r^{1+\alpha}, \text{\ \ for all $r$ small enough.}
\]
We call $z_0$ a \emph{measurable deep point} in $E$, if there exists constants $M, \beta, \delta>0$ such that
\[
  \Ar(B(z_0, r) \setminus E) \leq M \, r^{2+\beta}, \text{ whenever  } r<\delta .
\]
\end{definition}

It is obvious that measurable deep implies deep. Observe also that if $E\subset E'$ are Borel sets and $z_0$ is a measurable deep point in E then $z_0$ is a measurable deep point in $E'$.

\begin{theorem}[{\cite[Cor.~4.5]{mcm}}]\label{deep}
  Let $\epsilon > 0$ be arbitrary
  and define
  \[
    S_\epsilon 
    = \{ z \in K(P) : d(P_c^n(z), S) < \epsilon \text{ for all }n \geq 0  \} .
  \]
  Then $\omega $ is a measurable deep point in $S_\epsilon$. 
\end{theorem}

If $\phi : \RS \to \RS$ is quasiconformal and the support of $\mu_\phi$ gets thin close to a point $z_0$, we can expect $\phi$ to be regular at $z_0$.
There are several results in this direction (see e.g. \cite[Chapter 6]{lehvir}), and we will find use for the following theorem of McMullen.

\begin{theorem}[{\cite[Theorem 2.19]{mcm2}}]\label{reg}
  Suppose $\phi : U \to V$ is quasiconformal and let $\Omega = U \setminus \supp \mu_\phi$.
  If $z_0$ is a measurable deep point of $\Omega$, then $\phi$ is $C^{1+\alpha}$ at $z_0$  for some $\alpha>0$, and $\phi'(z_0) \neq 0$ .
\end{theorem}
\begin{remark}
The last conclusion, i.e. the derivative being nonzero, is implicitely used in McMullen's text although not explicitely stated. For completeness, let us show how it follows from the condition of measurable depth. In \cite[Lemma 6.1]{lehvir} it is proven that if the dilatation of $\phi$, say $D(z)$,  satisfies that the integral
\[
\int\int_{|z|<r} \frac{D(z)-1}{|z|^2} d\sigma
\]
is convergent for every $r$, then $\phi$ is complex differentiable at $0$ with nonzero derivative. By breaking the disk of radius $r$ into a series of annuli $A_n=\{\frac{1}{2^{n+1}} < |z| < \frac{1}{2^n}\}$, one can bound the integral on each annulus by $\frac{4 (K-1) M}{2^{\beta n}}$, where $M$ and $\beta$ are the constants given by the fact  that $0$ is a measurable deep point of $\Omega$, and $K$ is the bound on the dilatation of $\mu$. The integral is thus bounded from above by a geometric series and hence finite.
\end{remark}

McMullen also showed that when $\theta$ is a quadratic irrational, the Siegel disk is self-similar at the critical point $\omega$. More precisely he showed the following.

\begin{theorem}[{\cite[Theorem7.1]{mcm}}] \label{selfsim}
Suppose $\theta$ is a quadratic irrational, and let $s$ denote the periodicity of the coefficients of its continued fraction. Let $P$ be as in (\ref{polyn}) and $S$ be its Siegel disk. 
Then, there exist $\alpha>0$ and a locally defined homeomorphism $\psi$, conjugating 
$P^{q_n}$ to $P^{q_{n+s}}$ on $\partial S$ for $n$ sufficiently large.
More precisely, we have
\[
  \psi(z) = \begin{cases}
    \omega + \kappa (z-\omega) + O( (z-\omega)^{1+\alpha} ) &  \text{ if $s$ even} \\
    \omega + \kappa (\overline{z-\omega}) + O( (z-\omega)^{1+\alpha} ) & \text{ if $s$ odd}
  \end{cases}
\]
for some complex number $\kappa$ with $ 0 < |\kappa| < 1$. 
\end{theorem}

It follows from the theorem, that when $s$ is even, $\psi$ is $C^{1+\alpha}$-conformal at $\omega$, and when $s$ is odd, $\psi$ is $C^{1+\alpha}$-anticonformal. 

\begin{remark}[Tight self-similarity of $S$] \label{tightselfsim}
We will see later (see Remark \ref{tightselfsim2})  that this implies that the Siegel disk is actually tightly self-similar. More precisely, if $s$ is even, then
\[
  \partial S - \omega \tsim{0} \kappa (\partial S - \omega), 
\]
and, when $s$ is odd, 
\[
  \partial S - \omega \tsim{0} \kappa (\overline{\partial S - \omega}),
\]
where $\kappa$ is the scaling factor in Theorem \ref{selfsim}.
\end{remark}

With the preceeding theorem in hand, Buff and Henriksen \cite{MR1713131} were able to prove that for some values of $\theta$, such as the golden mean $\frac{\sqrt{5}-1}{2}$, the Siegel disk $S$ contains an Euclidean triangle with a vertex at the critical point.

\section{Quasiconformal surgery and Proof of Theorem A} \label{surg}

In this section we prove Theorem A.
We shall see that the proofs of (a), (b) and (c). follow directly from a surgery construction due to Shishikura, whereas the last part can be derived by bounding the relative area of the support of the quasiconformal distorsion of $\Phi$ as we approach $\omega$. The main idea of the surgery is simply to replace the dynamics in the hole of the Herman ring with an irrational rotation. In this way we obtain a quasiregular map $F$, which is quasiconformally conjugate to $P$. Letting $\Phi$ denote the conjugacy from $P$ to $F$, we then check that it has the stated properties. Details are as follows (c.f. \cite{shi} and \cite[Section 7.2]{surbook}).

We keep the notation from the setup in Section \ref{setup}.  Let $\phi_0 : H \to \{ z: r < |z| < 1 \} $ denote the linearizing map, conjugating $f$ to $R_\theta : z \mapsto e^{2i\pi\theta}z$.
Define three topological disks $U_1 \Subset U_2 \Subset U_3$, such that $\partial U_1 \subset H$ and $\partial U_2 \subset H$ are $f-$invariant curves and $U_3$ is the polynomially convex hull of $H$ which in this case is the complement of the unbounded component of the complement of $H$.
The image under $\phi_0$ of $\partial U_1$ is a circle, i.e., the boundary of a disk $V_1$.
Similarly, we define $V_2$ to be the disk whose boundary is $\phi_0(\partial U_2)$, and we let $V_3 = B(0, 1)$.
See Figure \ref{figsur}.

\begin{figure}[hbt!]
\captionsetup{width=0.8\textwidth}
  \begin{center}
  \includegraphics[width=0.9\textwidth]{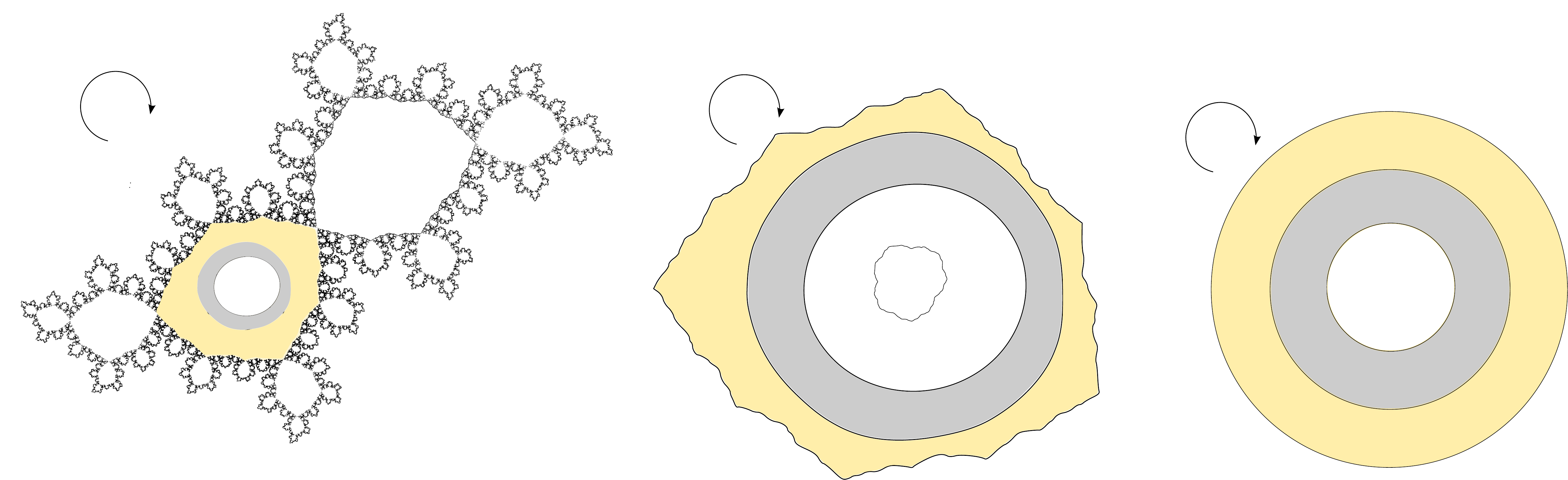}     
  \setlength{\unitlength}{\textwidth}
  \put(-0.64,0.1){\vector(1,0){0.1}}
  \put(-0.59,0.11){$\Phi$}
  \put(-0.54,0.09){\vector(-1,0){0.1}}
  \put(-0.59,0.065){$\Psi$}    
  \put(-0.26,0.1){\vector(1,0){0.05}}
  \put(-0.25,0.11){$\phi_1$}
  \put(-0.23,0.225){$R_\theta$}
   \put(-0.1,0.187){\small $V_3$}
    \put(-0.1,0.155){\small $V_2$}
   \put(-0.1,0.12){\small $V_1$}
\put(-0.49,0.237){$F$}
\put(-0.3,0.05){\scriptsize $U_3$}
\put(-0.34,0.1){\scriptsize $U_1$}
\put(-0.318,0.075){\scriptsize $U_2$}
  \put(-0.85,0.237){$P$}  
  \end{center}
  \caption{\small Quasiconformal surgery to produce the polynomial $P$ with a Siegel disk starting from the rational map $f$ with the Herman ring $H$.} 
    \label{figsur}
\end{figure}
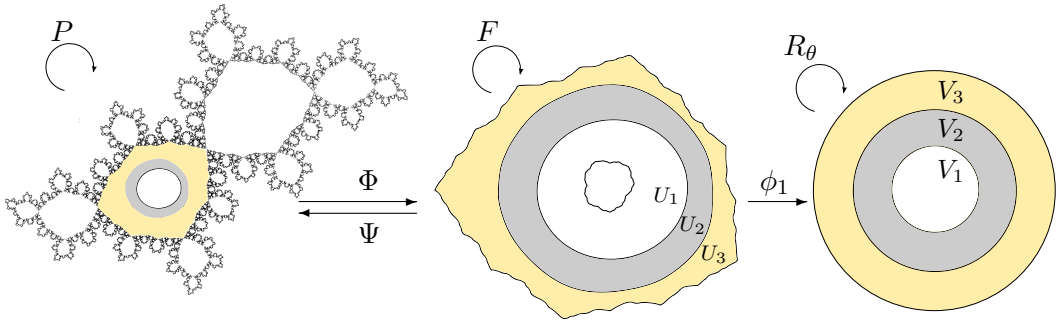

We modify and extend $\phi_0$ as to define it on all of $U_3$. Define $\phi_1 : U_3 \to V_3$ by letting it be equal to $\phi_0$ on $U_3\setminus U_2$, by requiring that it maps $(U_1, 0)$ conformally to $(V_1, 0)$ and interpolating quasiconformally on the annulus $U_2 \setminus U_1$, so we get a quasiconformal mapping $\phi_1 : U_3 \to V_3$, which we shall use to paste the rigid rotation into the Herman ring. 

We have set up the machinery to plug the hole of the Herman ring.
Let
\[
  F := \begin{cases}
    f & \text{ on } \RS \setminus U_2 \\
    \phi_1^{-1} \circ R_\theta \circ \phi_1 & \text{ on } U_2 .
  \end{cases}
\]
This is a model of a quadratic polynomial, since the pole of $f$ no longer exists and the global degree is now two. It is, however, only quasiregular.
To remedy this, we will define an $F-$invariant Beltrami coefficient $\mu$ with bounded dilatation with the intention of applying the Measurable Riemann Mapping Theorem.  This Beltrami coefficient is defined by pieces. We start by defining it in $U_3$, by pulling back the standard Beltrami coefficient $\mu_0=0$ under $\phi_1$, that is $\mu = \phi_1^* (0)$ on $U_3$ or equivalently, $\mu(z)=\bar{\partial} \phi_1 / \partial \phi_1$.
Observe that $\mu$  is invariant by $F|_{U_3}$ by construction, and it has bounded dilatation, precisely that of $\phi_1$. 

We can extend $\mu$ recursively to the backward orbit of $U_3$, by letting $\mu := (F^n)^*(\mu) $ on $F^{-n}(U_3)$, for every $n\geq 1$.
Finally we can extend it to all of $\RS$ by letting $\mu$ vanish outside the backwards orbit of $U_3$.

Since $\mu$ is invariant by $F|_{U_3}$, the extension is invariant by $F$.
Also, since $F$ is analytic outside $\overline{U}_2$ it follows that 
$||\mu||_\infty = ||\mu_{|_{U_3}}||_\infty=: k < 1$.
Hence, $F$ is holomorphic with respect to the almost complex structure defined by $\mu$, and therefore we can apply the Measurable Riemann Mapping theorem (see e.g. \cite[Theorem 1.27]{surbook}) to obtain a 
a quasiconformal homeomorphism 
$\Psi = \Psi_\mu : \RS \to \RS$,
satisfying $\mu  = \Psi^*(0)$ or, equivalently, $\bar{\partial} \Psi = \mu \, \partial \Psi$.
Since $\Psi$ is unique up to fixing the image of three points, we normalize it by requiring
$\Psi (\omega_2) = \omega$,
$\Psi (\phi_1^{-1}(0)) = 0$,
and $\Psi(\infty) = \infty$.

The map  $Q = \Psi \circ F \circ \Psi^{-1}$ is a quadratic polynomial.
Indeed, it is a holomorphic map of degree two having a fixed critical point at infinity. Observe that $0$ is fixed by $Q$. Since $\phi_1$ is conformal on $U_1$, so is $F$, and $F'(0)=R_\theta'(0)=e^{2\pi i \theta}$. This implies that $\Psi$ is also conformal on $U_1$ and hence $Q'(0)=F'(0)= e^{2\pi i \theta}$. Hence $z=0$ is a Siegel point of $Q$.  Additionally, $\omega=-e^{2\pi i \theta}/2$ is a critical point, thus we conclude that $Q=P$.

Let $\Phi := \Psi^{-1}$  and $D:= \Psi(U_2)$ (see Figure \ref{figsur}). We now show that $\Phi$ satisfies properties (a) to (d). in Theorem A.

{\em Property} (a). First notice that $\phi_1 \circ \Phi$ is conformal on $\Psi(U_3)$ and conjugates $P$ to $R_\theta$ on this domain. So $\Psi(U_3) \subset S$.
By maximality of $H$, $\Psi(U_3) = S$, because if $\Psi(U_3)$ were only a subdisk of $S$ then $H$ would not  the maximal domain of linearization.
By construction, $\Phi$ conjugates $P$ to $F$ everywhere, but since $F = f$ except on $U_2$, $\Phi$ conjugates $P$ to $f$ everywhere except on $D$. 
Finally, since $U_2 \Subset U_3$ we know $D \Subset S$.

{\em Property} (b).
Recall that $A_P(\infty)$ denotes the basin of attraction of infinity of $P$.
We have $\partial A_P = J(P)$.
On $A_P(\infty)$, $\Phi$ conjugates $P$ to $f$. Hence $\Phi(A_P(\infty)) \subset A_f(\infty)$.
Similarly $\Psi$ conjugates $f$ to $P$ on $A_f(\infty)$, and therefore
$A_f(\infty) \subset \Phi(A_P(\infty))$. Hence $\Phi(A_P(\infty)) = A_f(\infty)$, and
$\Phi(J(P)) = \partial \Phi(A_P(\infty)) = \partial A_f(\infty)$.
We have already seen that $S = \Psi(U_3)$.
Hence $\Phi(\partial S) = \partial U_3 = \partial H_2$.
That $\Phi(\omega) = \omega_2$ is evident from the normalization of $\Psi$.

{\em Property} (c).  By construction, $\bar{\partial}\Psi = 0$ on the complement of
$\widehat{U}=\cup_{n=0}^\infty f^{-n}(U_2)$, and thus locally conformal on the complement of the closure of this set.
Hence the inverse map $\Phi$ is locally conformal on the complement of the closure of $\GOD = \Psi(\widehat{U})$.
The boundary of $\GOD$ consists of a countable number of closed real analytic curves as well as the boundary of $J(P)$.
By Petersen's theorem, Theorem \ref{Petersen}, the boundary of $\GOD$ has measure zero.
Thus we can conclude $\bar{\partial} \Phi = 0$ almost everywhere on the complement of $\GOD$.

{\em Property} (d).
In view of Theorem \ref{reg}, it is enough to prove that $\omega$ is a measurable deep point  in the complement of $\GOD$.
To show this, McMullen has done the heavy lifting by proving Theorem \ref{deep}: $\omega$ is a measurable deep point in $S_\epsilon$.  Since $D$ is a definite distance away from $\omega$ and again using the $J(P)$ has measure zero, we deduce that $\omega$ is a measurable deep point of $S_\epsilon \setminus \left( D \cup J(f) \right)$.
Hence it is enough to show that $S_\epsilon \setminus \left( D \cup J(f) \right) \subset \left(\C\setminus \GOD\right)$.

The invariant set $D$ has two preimages; $D$ itself and another one $D'$.
Choosing $\epsilon$ small enough, we can assume $S_\epsilon$ does not meet $D'$.
Any point $z \in S_\epsilon \setminus \left( D \cup J(f) \right)$ lies in the Fatou set, and is thus eventually mapped into $S$.
Clearly, $z$ is not an element of $D \cup D'$.
However, since $S_\epsilon$ is forward invariant, none of the iterates of $P$ can map $z$ into $D'$.
So $z$ is not an element of $\GOD$. This concludes the proof of Theorem A.

In the course of the proof, we showed that $\omega$ is a measurable deep point $\C \setminus \GOD$.
We shall use this later, so we formally state it.

\begin{lemma} \label{omegadeep}
  The critical point $\omega$ is a measurable deep point of $\C \setminus \GOD$. 
\end{lemma}

\section{Preliminaries about $C^{1+\alpha}$-conformal mappings and tightly similar sets}
To prove Theorem B, we need to establish some elementary properties of $C^{1+\alpha}$-conformal homeomorphisms and tightly similar sets.
We see in this section that the two notions complement each other well.

First we prove that $C^{1+\alpha}$ regularity extend to inverses when the map in question is quasiconformal.

\begin{proposition} \label{inverse}
  Let $\phi : U \to V$ be a homeomorphism between the open sets $U$ and $V$, and suppose $\phi$ is $C^{1+\alpha}$-conformal at $z_0 \in U$ with  $\phi'(z_0) \neq 0$.  Then $\phi^{-1}$ is $C^{1+\alpha}$-conformal at $\phi(z_0)$.
\end{proposition}

\begin{proof}
We can assume that $\phi'(0) = 1$, so by hypothesis
\[
  \phi(z) = z + R(z), \text{ where } |R(z)| < M |z|^{1+\alpha} \text{ when } |z| < \delta
\]
for some $M$ and $\delta > 0$.

When $|z|$ is small enough, $|\phi(z)| \geq |z|/2$.
Hence there exists $\delta' > 0$ so that when $w \in B(0, \delta')$, $|\phi^{-1}(w)| \leq 2|w|$ and $\phi^{-1}(w) \in B(0, \delta)$.
Since $w = \phi(\phi^{-1}(w))$ we can write
\[ w= \phi^{-1}(w) + R(\phi^{-1}(w))\]
and hence we have 
\[ \phi^{-1}(w) = w - R(\phi^{-1}(w)) = w + \tilde{R}(w),
\]
with 
\[
|\tilde{R}(w)| = |R(\phi^{-1}(w)) | \leq M |\phi^{-1}(w)|^{1+\alpha} \leq M |2 w|^{1+\alpha}
\]
when $w \in B(0, \delta')$, which concludes the proof. \qedprop
\end{proof}

\begin{proposition} \label{nested}
  Let $A \subset B \subset C$ be compact sets.
  Then $A \tsim{z_0} C$ implies $A \tsim{z_0} B \tsim{z_0} C$.
\end{proposition}

\begin{proof} Let $\beta, \delta,$ and $K$ be such that 1. and 2. in definition \ref{tight} holds for $A$ and $C$.

To prove $A \tsim{z_0} B$, it is enough to prove that when $b \in B \cap B(z_0, \delta)$, we can find $a \in A$ such that $|a - b| \leq K |b-z_0|^{1+\beta}$.
But since $b \in C$, and $A \tsim{z_0} C$, we can find such $a$.

To prove $B \tsim{z_0} C$, it is enough to prove that when $C \in C \cap B(z_0, \delta)$, we can find $b \in B$ such that $|c - b| \leq K |c-z_0|^{1+\beta}$.
Since $b \in C$, and $A \tsim{z_0} C$, we can find such $b \in A \subset B$, satisfying the  inquality. \qedprop
\end{proof}

We end this section by seeing that tight similarity at a point is preserved by homeomorphisms which are $C^{1+\alpha}$-conformal at the point.

\begin{proposition} \label{preservesim}
  Let $A$ be a compact set and $\phi$ a homeomorphism which is $C^{1+\alpha}$-conformal at $z_0$.
  \begin{enumerate}
    \item If $\phi(z_0)=z_0$, $\phi'(z_0)=1$ then $A \tsim{z_0} \phi(A)$.
    \item If $\phi'(z_0) \neq 0$ and $B$ is a compact such that $A \tsim{z_0} B$, then $\phi(A) \tsim{\phi(z_0)} \phi(B)$.
  \end{enumerate}
\end{proposition}

\begin{proof}
To see the first statement, notice that for any $a \in A \cap B(z_0, \delta)$, we have 
$ \phi(a)=  a  + R(a)$. Hence
\[
|a - \phi(a)| \leq M |a-z_0|^{1+\alpha}
\]
for some $M>0$. This is the first of the requirements of Definition \ref{tight}. 
The second requirement follows by applying the same argument to $\phi^{-1}$, which is $C^{1+\alpha}$-conformal by Proposition \ref{inverse}.

We can deduce the second statement from the first.
Since scaling $A$ and $B$ around $z_0$ does not change whether they are tightly similar or not, we can suppose $\phi'(z_0)=1$.
Then
\[
  \phi(A) \tsim{z_0} A \tsim{z_0} B \tsim{z_0} \phi(B) ,
\]
and the two image sets are tightly similar. \qedprop
\end{proof}

\begin{remark}\label{tightselfsim2}
From this proposition, together with Theorem  \ref{selfsim} , we can see that $S$ is tightly self similar if $\theta$ is a quadratic irrational. Indeed, let $s$ be the period of the continued fraction, and assume it is even. Let $\psi$ be the homeomorpfism and $\kappa$ the scaling factor in Theorem \ref{selfsim}, which ensure the self-similarity of $S$. Then, the map 
\[
h(z) = (1/\kappa) (\psi(z+\omega) - \omega)
\]
sends $\kappa (\partial S-\omega)$ to $\partial S - \omega$,
and is of the form $h(z) = z + O(|z|^{1+\alpha})$. It follows from Proposition \ref{preservesim} 
that  $\partial S - \omega \tsim{0} \kappa(\partial S - \omega)$.
When $s$ is odd, $h(z) = (1/\kappa) (\overline{\psi(z + \omega) - \omega})$ gives the tight self-similarity between 
$\partial S - \omega$ and  $\kappa (\partial S - \omega)^*$,
where $(\partial S - \omega)^*$ denotes the set obtained from first translating $\partial S$ and then reflecting in the real axis. 
\end{remark}

\section{Geometry invariance.\\Proof of Theorem B and Corollary \ref{extended}}
In this section we prove Theorem B.
Most of the statements are fairly obvious by the results we have uncovered by now.
The delicate part is to prove that $\partial A_f(\infty) \tsim{\omega_2} J(f)$.

Recall that $\Phi$ denotes the quasiconformal homeomorphism given in Theorem A, conjugating $P$ to $f$ on $\widehat{\mathbb C} \setminus D$, which is conformal outside $\mathcal{D}=\cup_{n=0}^\infty P^{-n}(D)$. 
Our strategy will be to prove that $J(P) \tsim \omega \overline{\GOD}$.
Since $\Phi$ is $C^{1+\alpha}$ -conformal at $\omega$, $\Phi(J(P)) = \partial A_f(\infty)$, and we shall see that $\Phi(\overline{\GOD})$ contains $J(f)$, this will be sufficient to prove the desired similarity by Propositions \ref{nested} and \ref{preservesim}.

The idea in proving $J(P) \tsim{\omega} \overline{\GOD}$ is that the area of components of $\GOD$ must quickly decrease as we approach $\omega$.
By bounded geometry, this means that the diameter of the components of $\GOD$ must also quickly decrease.
Each connected component $D'$ of $\GOD$ is contained in a connected component $F'$ of the Fatou set $F(P)$.
The modulus of the annulus $F' \setminus D'$ is independent of which component $D'$ we are considering and so for any $z \in D'$ the distance $d(z, \partial F')$ is comparable to the diameter of $D'$, i.e., quickly decreasing as we approach $\omega$. This is the idea, it remains to fill in the details.

\begin{definition}\label{inout}
Let $U$ be a bounded simply connected domain and $z$ a point in $U$.
The \emph{inner radius} $\inR (U, z)$ is given by
\[
  \sup \left\{ r : B(z, r) \subset U \right\}
\]
whereas the \emph{outer radius} $\outR (U, z)$ is
\[
  \inf \left\{ r: U \subset B(z, r) \right\}
\]
\end{definition}

\begin{trivlist}
\item {\em Claim 1.}
There exists $K_1 \in (1, \infty)$ such that for any connected component $D'$ of $\GOD$,
$$
1 \leq \frac{ \outR(D', \alpha') } { \inR(D', \alpha') } \leq K_1,
$$
where $\alpha'$ denotes the preimage of the Siegel point lying in $D'$.
\end{trivlist}

The claim follows from Koebe's Distortion Theorem (see e.g. \cite{carlesongamelin}) .
Let $\phi : S \to B(0, 1)$ denote a map linearizing $P$, and $F'$ the component of the Fatou set $F(P)$ that contains $D'$.
When $n$ is sufficiently large, $\phi \circ P^n : F' \to B(0, 1)$ is a conformal isomorphism and, by construction of $D'$ (see Section \ref{surg}), it maps $D'$ onto a round disk $B(0, s)$, and $\alpha'$ to the origin.
Applying Koebe's Distortion Theorem to the inverse map, we get
\[ 
  \frac{ \outR(D', \alpha') } { \inR(D', \alpha') } 
  \leq \left( \frac{1+s}{1-s} \right)^2  =:K_1. 
\]
  
\begin{trivlist}
\item {\em Claim 2.}
There exists $K_2, K_3 \in (0, \infty)$ such that for any connected component $D'$ of $\GOD$,
\[
  K_2 \sqrt{\Ar D'} \leq \outR(D', \alpha') \leq K_3 \sqrt{\Ar D'}
\]
where $\alpha' \in D'$ is eventually mapped to the Siegel point.
\end{trivlist}
The claim is trivial, since $\pi$ times in-radius squared gives a lower bound on the area, $\pi$ times the out-radius squared gives an upper bound, and in-radius and out-radius are comparable by Claim 1.

\begin{trivlist}
\item {\em Claim 3.}
There exists $K_4 \in (2, \infty)$ such that for any connected component $D'$ of $\GOD$, and any $z \in D'$, 
\[d(z, J(P)) < K_4 \outR(D', \alpha'), \]
where $\alpha'$ is the point in $D'$ that is eventually mapped to the Siegel point.
\end{trivlist}
This claim is a consequence of Gr\"oztsch's inequality.
Letting $F'$ denote the component of the Fatou set that contains $D'$, the modulus of $F' \setminus \overline{D}'$ is equal to the modulus $m$ of $S \setminus \overline{D}$.
Put $\outR = \outR(D', \alpha')$.
There exists a constant $c >1$ only depending on $m$, such that the modulus of $B(\alpha', c\outR) \setminus \overline{B}(\alpha', \outR)$ equals $2m$ (we could have chosen any other number larger than $m$).
If $B(\alpha', c\outR)$ were  contained in $F'$, then $B(\alpha', c\outR) \setminus \overline{B}(\alpha', \outR)$ would be contained in $F' \setminus \overline{D}'$ in contradiction with Gr\"oztsch's inequality.
So there exists $z' \in \partial F' \subset J(P)$ with $|z' - \alpha'| \leq c\outR$.
By the triangle inequality, $|z - z'| \leq |z - \alpha'| + |\alpha' - z'| \leq \outR + c\outR$, which proves the claim, with $K_4:=c+1$.

\vspace{0.3cm}

Lemma \ref{omegadeep} tells us that $\omega$ is measurable deep in the complement of $\GOD$, i.e. there exist constants $\beta, \delta > 0$ and $M$ so $\Ar (\GOD \cap B(\omega, r)) \leq M r^{2+\beta}$, when $r< \delta$.
Pick $r_0 > 0$ small enough so that
\begin{itemize}
  \item $r_0 < \frac{\delta}{4}$, and
  \item letting $\gamma = \left( 1+\beta/2 \right)^{-1} < 1$, 
      $2  \left(K_3 \sqrt{M}\right)^\gamma r \leq \frac{1}{2} r^\gamma$, when $r \leq r_0$.
\end{itemize}

\begin{trivlist}
\item {\em Claim 4.}
There exists $\delta', \beta > 0$ and $K$, so that 
\[
d(z, J(P)) \leq K |z-\omega|^{1+\beta/2}
\]
 whenever $z \in \GOD \cap B(z_0, \delta)$.
\end{trivlist}
There are only finitely many components $D'$ of $\GOD$ with $\outR(D', \alpha') > r_0$.
If there were infinitely many, then by Claim 2, there area of $\GOD$ would be infinite,
but $\GOD$ is contained in the filled-in Julia set of $P$, and it is well known that the latter set is contained in $B(\omega, 2)$. Hence we can pick $\delta' > 0$ so that
\begin{itemize}
  \item $\delta' \leq \delta/2$
  \item $B(\omega, \delta')$ does not meet a  component of $\GOD$ with $\outR(D', \alpha') > r_0$.
\end{itemize}

Consider an arbitrary $z \in \GOD \cap B(\omega, \delta')$.
Such $z$ lies in a component $D'$ of $\GOD$, and we have chosen $\delta'$ so that $r := \outR(D',\alpha') \leq r_0$.
Notice $D' \subset B(\omega, |z-\omega| + 2 r)$.
As $|z-\omega| < \delta' \leq  \delta/2$, and $2 r \leq 2 r_0 \leq \delta'$, it holds that $D' \subset B(\omega, \delta)$.
So 
$$
\Ar D' \leq \Ar \GOD \cap B(\omega, |z-\omega|+2r) \leq M (|z - \omega| + 2r)^{2+\beta}.
$$
Using Claim 2, we get
\[
  r \leq K_3 \sqrt{M} \left( |z-\omega|+2r \right)^{1+\beta/2}
  \Leftrightarrow 
  r^\gamma - 2 r \left(K_3 \sqrt{M}\right)^\gamma  \leq  \left(K_3 \sqrt{M}\right)^\gamma |z-\omega| .
\]
Since $r \leq r_0$,  and by the choice of $r_0$, the left hand side is bounded from below by $\frac{r^\gamma}{2}$, and we get that
\[
  r \leq K_5 |z - \omega|^{1+\beta/2},
\]
for some $K_5:= \in (0, \infty)$.
By Claim 3, $d(z, J(P)) \leq K_3 K_5  |z - \omega|^{1+\beta/2}$ which proves Claim 4.

\begin{trivlist}
\item {\em Claim 5.}
$J(P) \tsim{\omega} \overline{\GOD}$.
\end{trivlist}
We have just established one of the two requirements of Definition \ref{tight}.
The set $\overline{\GOD}$ is a completely invariant closed set contaning more than two points, so $J(P) \subset \overline{\GOD}$.
Hence, the other requirement is automatically satisfied.

\vspace{0.3cm}
Having proven Claim 5, we have navigated the rough part of the proof; the rest is smooth sailing.
By Theorem A, $\Phi(J(P)) = \partial A_f(\infty)$.
Again using Theorem A we obtain that  $L: = \Phi'(z_0) \neq 0$. Hence, Claim 5 together with Proposition \ref{preservesim} imply that 
\[
\Phi(J(P)) = \partial A_f(\infty) \tsim{\omega_2} \Phi(\overline{\GOD}).
\]
Let $h(z) := z/L+\omega$, and $g(z) := z - \omega_2$. Then, 
the map $\Psi:=g \circ \Phi \circ h$ maps $L(J(P) - \omega)$ onto $\partial A_f(\infty) - \omega_2$.
Moreover $\Psi$  maps the origin to itself, is $C^{1+\alpha}$  and has derivative $1$ at this point. 
So, in view of Proposition \ref{preservesim},
 \[
 L(J(P) - \omega) \tsim{0} \partial A_f(\infty) - \omega_2.
 \]


 Let us now see  that $\partial A_f(\infty) \subset J(f) \subset \Phi(\overline{\GOD})$.
The first inclusion is immediate.
If $z \in J(f) \setminus \Phi(\GOD)$, then $u:=\Phi^{-1}(z)$ has a bounded forward orbit that avoids $D$.
Since $u$ can never be mapped to $S \setminus D$ (otherwise $z$ would eventually be mapped to $H$), we must have $\Phi^{-1}(z) \in J(P) \subset \overline{\GOD}$.
Hence $z \in \Phi(\overline{\GOD})$.

Since $\partial A_f(\infty) \subset J(f) \subset \Phi(\overline{\GOD})$, and $A_f(\infty) \tsim{\omega_2} \Phi(\overline{\GOD})$, Proposition \ref{nested} shows that $\partial A_f(\infty) \tsim{\omega_2} J(f)$. This finishes the proof of Theorem B, statement (a).

\vspace{0.3cm}

To see part (b), notice that $\partial S$ is self-similar by Remark \ref{tightselfsim}.
The mapping $\psi$ of Theorem \ref{selfsim} can be turned into mapping of $\partial H_2$ by letting $\tilde{\psi}=\Phi \circ \psi \circ \Phi^{-1}$.
By Propostion \ref{inverse} this composition is $C^{1+\alpha}$-conformal or anticonformal, and has the same conformal or anticonformal derivative $\kappa$ as $\psi$.
We can conclude that $\partial H_2$ is self-similar in the sense of McMullen.

For part (c), assume that S contains an open triangle $T$ with vertices $a$, $b$ and $\omega$.
Shrinking $T$ we can assume $T \cap \partial S = \left\{ \omega \right\}$.
We know from Theorem A (d), that $\Phi$ is $\C$-differentiable at $\omega$ and  $\Phi'(\omega) \neq 0$. Moreover, by Proposition \ref{inverse}, we have that $\Phi^{-1}$ is also $\C$-differentiable at $\omega_2$.
The image $\Phi([\omega, a])$ is a curve in $H$ having a tangent at the starting point, and $\Phi([\omega, b])$ is a curve in $H$ having a tangent at the starting point.
The angle between the two tangents is the same as the angle of $T$ at $\omega$, since $\Phi'(\omega)$ exists and is non-vanishing.
Thus a line segment in the gap between the the two tangents, will not intersect the two image curves in a small enough neighborhood of $\omega_2$.
Hence there is room for a triangle in $H$ with a vertex at $\omega_2$.
As the inverse map $\Phi^{-1}$ is $\C$-differentiable at $\omega_2$ with non-vanishing derivative, we can use the exact same argument to establish that the existence of a triangle in $H$ with a vertex at $\omega_2$ implies the existence of a triangle in $S$ with a vertex at $\omega$.
\qedprop

\subsubsection*{Proof of Corollary \ref{extended}}
It only remains to prove Corollary \ref{extended}.
The statements of Theorem A and B hold true if we replace $\omega_2$ with $\omega_1$.
When $f$ is viewed as a mapping of the sphere, $\omega_1$ and $\omega_2$ do not play different roles.
If we change coordinates by the map $z \mapsto 1/z$, then $f$ takes the form $g(z) = 1/f(1/z) = b^{-1}\frac{az+1}{z+a}$,
i.e. we get a mapping of the form covered by Theorems A and B.
The change of coordinates interchanges outer and inner boundary of $H$ and $\omega_1$ and $\omega_2$.
More precisely, $z \mapsto 1/z$ maps $\omega_1$ for $f$ to $\omega_2$ for $g$, and $\omega_2$ for $f$ to $\omega_1$ for $g$.
Theorems A and B for $g$  use the change of coordinates to get the approiate statements for $\omega_1$.
In particular, when proving the corollary, we can suppose $f^m(u) = \omega_2$.

Let $k \geq 0$ and $v$ be such that $P^k(v) = \omega$.
Let us first see that there exists an inverse branch of $P$  mapping a neighborhood of $\omega$ conformally onto a neighborhood of $v$. 
Indeed, the boundary of $S$ is included in the accumulation set of the forward orbit of $\omega$ which must be infinite.
So $P^j(v) \neq \omega$, for $j = 0, 1, \ldots, k-1$.
Since $\omega$ is the only zero of $P'$, this implies in particular that $(P^k)'(v) \neq 0$, which guarantees the existence of the inverse branch.

The same argument shows that if $k \geq 0$ and $f^k(v) = \omega_2$, then there exists an invere branch of $f^k$ mapping a neighborhood of $\omega_2$ onto a neighborhood of $v$.

Part (a). We know that $\Phi$ conjugates $P$ to $f$ on $\RS \setminus D$.
Hence $\Phi$ conjugates $P^n$ to $f^n$ on $\RS \setminus \cup_{j=0}^{n-1}(D)$.
That means $\Phi = f^{-n} \circ \Phi \circ P^n$ in a neighborhood of $\omega'$, for a suitably chosen branch of $f^{-n}$.
Since the property of being $C^{1+\alpha}$ at a point is clearly preserved by compostion with conformal mappings, we get that  $f^{-n} \circ \Phi \circ P^n$ is $C^{1+\alpha}$-conformal at $\omega'$.

Part (b). Using a suitably chosen branch of $P^{-n}$ and Propostion \ref{preservesim}, we get 
\[
J(P) -\omega \tsim{0} J(P) - \omega'.
\]
Using a suitably chosen branch of $f^{-m}$ in the same way, we get 
\[
J(f) - \omega_2 \tsim{0} J(f) - u.
\]

Part (c). There exists an inverse branch $f^{-m}$ that maps $\omega_2$ to $u$, and maps the part of the boundary of $H$ that lies in some neighborhood of $\omega_2$ onto the part of the boundary of $H'$ that lies in some neighborhood of $u$.
Using this inverse branch, we can tranport the mapping $\phi$ expressing the self-similarity 
at $\omega_2$ into a map $\tilde{\phi}$ by conjugating with $f^{-n}$.

Part (d). Is also proven using that inverse branches are conformal isomorphism.
We leave the details to the reader.

\bibliographystyle{amsalpha}
\bibliography{hstruct}
\end{document}